\documentclass[12pt,a4paper]{amsart}
\usepackage[all,cmtip]{xy}
\usepackage{enumerate, comment}
\usepackage{ amssymb, latexsym, amsmath}
\usepackage{fullpage}
\usepackage{url}
\usepackage{hyperref}
\usepackage{ crimson }
\usepackage{amsfonts}
\usepackage{tikz-cd}
\usepackage{amssymb}
\usepackage{amsthm}
\usepackage{amsmath}
\usepackage{amscd}
\usepackage[mathscr]{eucal}
\usepackage{indentfirst}
\usepackage{graphicx}
\usepackage{graphics}
\usepackage{pict2e}
\usepackage{epic}
\usepackage[OT2,T1]{fontenc}
\numberwithin{equation}{section}
\usepackage[margin=2.9cm]{geometry}

\theoremstyle{plain}

\newtheorem{Th}{Theorem}[section]
\newtheorem{Lemma}[Th]{Lemma}
\newtheorem{Cor}[Th]{Corollary}
\newtheorem{Prop}[Th]{Proposition}

 \theoremstyle{definition}
\newtheorem{Def}[Th]{Definition}

\newtheorem{?}[Th]{Problem}

\newcommand{\Q}{\mathbb{Q}}
\newcommand{\A}{Ad^0\bar{\rho}}
\newcommand{\Z}{\mathbb{Z}}

\newcommand{\Hp}{H_{\Q(\mu_p)}'}
\newcommand{\F}{\mathbb{F}}
\newcommand{\D}{\text{Diag}(\bar{\rho})}
\DeclareSymbolFont{cyrletters}{OT2}{wncyr}{m}{n}

\DeclareFontFamily{U}{wncy}{}
    \DeclareFontShape{U}{wncy}{m}{n}{<->wncyr10}{}
    \DeclareSymbolFont{mcy}{U}{wncy}{m}{n}
    \DeclareMathSymbol{\Sh}{\mathord}{mcy}{"58}

\newcommand\mtx[4] { \left( {\begin{array}{cc}
   #1 & #2 \\
   #3 & #4 \\
  \end{array} } \right)}
\begin{document}

\title{Constructing Certain Special Analytic Galois Extensions}
  
\author{Anwesh Ray}
\address{Cornell University \\ Department of Mathematics \\
  Malott Hall, Ithaca, NY 14853-4201 USA} 
\email{ar2222@cornell.edu}
\maketitle
\begin{abstract}
For every prime $p\geq 5$ for which a certain condition on the class group $\text{Cl}(\Q(\mu_p))$ is satisfied, we construct a $p$-adic analytic Galois extension of the infinite cyclotomic extension $\Q(\mu_{p^{\infty}})$ with some special ramification properties. In greater detail, this extension is unramified at primes above $p$ and tamely ramified above finitely many rational primes and its Galois group over $\Q(\mu_{p^{\infty}})$ is isomorphic to a finite index subgroup of $\text{SL}_2(\Z_p)$ which contains the principal congruence subgroup. For the primes $107,139,271$ and $379$ such extensions were first constructed by Ohtani and Blondeau. The strategy for producing these special extensions at an abundant number of primes is through lifting two-dimensional reducible Galois representations which are diagonal when restricted to $p$.
\end{abstract}
\section{Introduction}
Let $p$ be a prime number, the tame Fontaine-Mazur conjecture (conjecture 5a in \cite{FontaineMazur}) posits that an infinite Galois extension of a number field whose Galois group over $\Q$ is isomorphic to a $p$-adic analytic group is either ramified at infinitely many primes or is infinitely ramified at a prime dividing $p$. It is natural to ask if such extensions exist once we pass up an infinite cyclotomic extension $\Q(\mu_{p^{\infty}})$. We show that there are abundantly many primes $p$ for which there exists such an extension of $\Q(\mu_{p^{\infty}})$ whose Galois group is isomorphic to a finite index subgroup of $\text{SL}_2(\Z_p)$ with tame ramification above finitely many rational primes and unramified at primes above $p$.
\par Let $p\geq 5$ be a prime and $\zeta_p=e^{\frac{2\pi i}{p}}$. Denote by $\Hp$ the $p$-Hilbert Class field of $\Q(\mu_p)$, we identify $\text{Gal}(\Hp/\Q(\mu_p))$ with the $p$-part of the class group $\text{Cl}(\Q(\mu_p))$.
\par We associate a Galois representation to class group data. Denote by \[\mathcal{C}:=\text{Cl}(\Q(\mu_p))\otimes \F_p\] and $\bar{\chi}$ the mod $p$ cyclotomic character.
The Galois module $\mathcal{C}$ decomposes into isotypic components
\[\mathcal{C}=\bigoplus_{i=0}^{p-2}\mathcal{C}(\bar{\chi}^i),\]
where 
\begin{equation}\label{isotypicdef}
    \mathcal{C}(\bar{\chi}^i)=\{x\in \mathcal{C}\mid g\cdot x=\bar{\chi}^i(g) x\text{ for all }g\in \text{Gal}(\Q(\mu_p)/\Q)\}.
\end{equation}
\par By class field theory $\F_p$-line in $\mathcal{C}(\bar{\chi}^i)$ gives rise to an extension $L/\Q$ contained in the $p$-Hilbert Class field $H'_{\Q(\mu_p)}$. Since such a line is Galois stable, the extension $L$ is Galois over $\Q$. Any choice of isomorphism $\text{Gal}(L/\Q(\mu_p))\xrightarrow{\sim} \F_p$ gives rise to an element \[\beta\in H^1\left(G_{\Q},\F_p(\bar{\chi}^i)\right)\simeq \text{Hom}\left(G_{\Q(\mu_p)},\F_p(\bar{\chi}^i)\right)^{\text{Gal}(\Q(\mu_p)/\Q)}.\] This class does not depend on the choice of isomorphism. The class $\beta$ coincides with a reducible Galois representation
\[\bar{\rho}:G_{\Q}\rightarrow \text{GL}_2(\F_p)\]
defined by $\bar{\rho}=\mtx{\bar{\chi}^i}{\beta}{0}{1}$. The representation satisfies a number of conditions:
\begin{enumerate}
\item\label{one}$\bar{\rho}_{\restriction G_{\Q(\mu_p)}}$ is indecomposable,
\item\label{two} $\bar{\rho}_{\restriction G_{\Q(\mu_p)}}=\mtx{1}{\beta}{0}{1}$ is unramified at every prime,
\item\label{three} the local Galois representation at $p$ splits into a sum of characters $\bar{\rho}_{\restriction G_p}\simeq\mtx{\bar{\chi}_{\restriction G_p}^i}{0}{0}{1}.$
\end{enumerate}
\par The extension $L\neq \Q(\mu_p)$ and thus the cohomology class $\beta$ is non-trivial, from which condition $\ref{one}$ follows. As $L$ is contained in the Hilbert class field of $\Q(\mu_p)$ every prime of $\Q(\mu_p)$ is unramified in $L$. Condition $\ref{two}$ follows from this. By class field theory the principal prime-ideal $(1-\zeta_p)$ is split in the Hilbert class field of $\Q(\mu_p)$ and thus in $L$. Let $E$ denote the completion of $\Q(\mu_p)$ at $(1-\zeta_p)$. One deduces that $\beta_{\restriction G_E}\in H^1(G_E, \F(\bar{\chi}^i))$ is trivial. We observe that the order of $\text{Gal}(E/\Q_p)$ is coprime to $p$. From a standard argument appealing to the vanishing of $H^1(\text{Gal}(E/\Q_p),\F(\bar{\chi}^i))$ and the inflation-restriction sequence it follows that $\beta_{\restriction G_p}=0$. Condition $\ref{three}$ follows as a consequence.
\par For an introduction to the deformation theory of Galois representations, the reader may consult \cite{mazur}. We examine deformations of $\bar{\rho}$ with some prescribed local properties. In particular, the local deformation condition at $p$ should ensure that on passing up the infinite cyclotomic extension $\Q(\mu_{p^{\infty}})$, our deformations are unramified at primes $\mathfrak{p}|p$.
\par Our construction is based on that of Hamblen and Ramakrishna who in \cite{hamblenramakrishna} deduce the weak form of Serre's conjecture for a reducible residual Galois representation which satisfies some further conditions. Their method is based on a local to global deformation theoretic argument. Implicit to this construction is a choice of a local deformation condition at each prime at which the residual representation is allowed to ramify. Deformations of the residual representation are to satisfy these local conditions. In particular, at $p$ there is a choice of a local deformation condition which satisfies some key properties, namely the deformation condition is \textit{liftable} and \textit{balanced} (see Definition $\ref{liftablebalanced}$). Hamblen and Ramakrishna prescribe the \textit{ordinary} deformation condition (cf. \cite{RamFM} and \cite{Taylor}) and we shall on the other hand be working with the \textit{diagonal} deformation condition (see Definition $\ref{extraordinary}$). This deformation condition is \textit{liftable} and \textit{balanced} as we shall show (see Proposition $\ref{tangentequality}$). We shall then deduce be able to deduce the following theorem.
\begin{Th}\label{1.1}
Let $p\geq 5$ be a prime and $\mathcal{C}:=\text{Cl}(\Q(\mu_p))\otimes \F_p$. Suppose that there exists an odd integer $i\neq \frac{p-1}{2}$ such that $2\leq i\leq p-3$ such that $\mathcal{C}(\bar{\chi}^i)\neq 0$. Let $\bar{\rho}=\mtx{\bar{\chi}^i}{\ast}{0}{1}$ be the Galois representation associated to an $\F_p$-line in $\mathcal{C}(\bar{\chi}^i)$. There exist infinitely many lifts $\rho$ of $\bar{\rho}$
\[ \begin{tikzpicture}[node distance = 2.0 cm, auto]
            \node(G) at (0,0) {$G_{\Q}$};
             \node (A) at (3,0){$\text{GL}_2(\F_p)$,};
             \node (B) at (3,2){$\text{GL}_2(\Z_p)$};
      \draw[->] (G) to node [swap]{$\bar{\rho}$} (A);
       \draw[->] (B) to node{} (A);
      \draw[->] (G) to node {$\rho$} (B);
      \end{tikzpicture}\]such that the following conditions are satisfied
\begin{itemize}
\item $\rho(G_{\Q(\mu_{p^{\infty}})})$ contains the principal congruence subgroup of $\text{SL}_2(\Z_p)$
\item the determinant of $\rho$ is $\chi^{i+p^2(p-1)}$
\item $\rho_{\restriction G_p}$ is a direct sum of characters
$\rho_{\restriction G_p}=\varphi_1\oplus \varphi_2$. In particular, $\rho$ is abelian at $I_p$.
\item $\rho$ is unramified outside a finite set of primes. 
\end{itemize}
\end{Th}

\par The lift $\rho$ gives rise to an extension $\Q(\rho)$ of $\Q(\mu_{p^{\infty}})$ which is taken to be the fixed field of $\text{ker}\rho \subset G_{\Q}$. This extension is unramified at primes above $p$. Furthermore, at a prime $l\neq p$, since the residual representation $\bar{\rho}$ is unramified, the lift $\rho$ is tamely ramified at all primes $l\in S/\{p\}$, we are left with the following result.
\begin{Cor}\label{main}
Suppose that $p\geq 5$ be a prime and $i\neq \frac{p-1}{2}$ an odd integer between $2\leq i\leq p-3$ for which the isotypic space $\mathcal{C}(\bar{\chi}^i)\neq 0$ (cf. $\ref{isotypicdef}$). There are infinitely many Galois extensions $L/\Q(\mu_{p^{\infty}})$ for which
\begin{itemize}
\item the Galois group $\text{Gal}(L/\Q(\mu_{p^{\infty}}))$ topologically isomorphic to a subgroup of $\text{SL}_2(\Z_p)$ which contains the principal congruence subgroup.
\item $L$ is unramified at primes above $p$ and ramified above finitely many rational primes at which it is tamely ramified.
\end{itemize}
\end{Cor}
\begin{Prop}\label{lastprop}
The conclusion of Corollory $\ref{main}$ is satisfied at any prime $p$ such that 
\begin{enumerate}
    \item $p\geq 5$,
    \item $p$ is irregular,
    \item $p\equiv 1\mod{4}$,
    \item $p$ does not divide the class number of the totally real subfield $\Q(\mu_p)^+\subset \Q(\mu_p)$, i.e. Vandiver's conjecture is satisfied at $p$.
\end{enumerate}
\end{Prop}
\par This seems to indicate that there are infinitely many primes at which $p$ the implication of Corollory $\ref{main}$ is satisfied.
\par Such $p$-adic extensions $L$ were first constructed by Ohtani \cite{ohtani} and Blondeau \cite{blondeau} and their methods relied on lifting suitable irreducible Galois representations which are \textit{extraordinary} at $p$. The construction in \cite{ohtani} and \cite{blondeau} relies on the existence of an eigenform $f$ with \textit{companion forms} (and thus extraordinary at $p$) such the image of the residual representation $\bar{\rho}_f$ contains $\text{SL}_2(\F_p)$. Computations for $p<3500$ show that there are precisely four primes $107, 139, 271$ and $379$ for which such an eigenform $f$ exists.

\subsection*{Acknowledgements}
I am very grateful to my advisor Ravi Ramakrishna for introducing me to his Galois-theoretic lifting method and for making helpful suggestions. I would also like to thank Christian Maire for many fruitful conversations. I would also like to thank the anonymous referee for some valuable suggestions.
\section{Summary of Notation}
\begin{itemize}
\item For all rational primes $v$ fix a choice of an algebraic closure of $\bar{\Q}_v$ and a choice of an embedding $\iota_v:\bar{\Q}\hookrightarrow \bar{\Q}_v$ extending $\Q\hookrightarrow \Q_v$. The absolute Galois group of $\Q_v$ is denoted by $G_v:=\text{Gal}(\bar{\Q}_v/\Q_v)$. The embedding $\iota_v$ coincides with an inclusion $G_v\hookrightarrow G_{\Q}$.
\item Let $p\geq 5$ be a prime and let $\text{Cl}(\Q(\mu_p))$ denote the class group of $\Q(\mu_p)$.
\item The $\F_p[G_{\Q}]$-module $\mathcal{C}:=\text{Cl}(\Q(\mu_p))\otimes \F_p$ decomposes into a direct sum of isotypic-components
$\mathcal{C}=\bigoplus_{i=0}^{p-2}\mathcal{C}(\bar{\chi}^i)$.
\item We assume that there exists an odd integer $i\neq (p-1)/2$ for which $2\leq i\leq p-3$ and such that the isotypic component $\mathcal{C}(\bar{\chi}^i)\neq 0$. The adjoint Galois module associated with $\bar{\rho}$ of Theorem $\ref{1.1}$ by $Ad^0\bar{\rho}$. It consists of matrices $\mtx{a}{b}{c}{-a}$ of trace zero over $\F_p$ on which $g\in G_{\Q}$ acts via conjugation by $\bar{\rho}(g)$.
\item Let $\mathfrak{e}$ be the matrix $\mathfrak{e}=\mtx{1}{0}{0}{-1}$. The subspace consisting of diagonal matrices is denoted by $\D:=\F_p \mathfrak{e}$.
\item The ring of dual numbers $\F_p[\epsilon]:=\F_p[T]/(T^2)$ and $\epsilon$ denotes $T$ mod $T^2$.
\item Let $\mathcal{C}_{\Z_p}$ be the category of coefficient rings over $\Z_p$. The objects of $\mathcal{C}_{\Z_p}$ consist of Noetherian local $\Z_p$-algebras $R$ with maximal ideal $\mathfrak{m}$ endowed with an isomorphism \[\phi:R/\mathfrak{m}\xrightarrow{\sim} \F_p.\]
A morphism of coefficient rings is a homomorphism of $\Z_p$-algebras compatible with residual isomorphisms. 
\end{itemize}
\section{Lifting to Characteristic Zero}
We proceed to describe the local deformation condition at $p$ which will in particular ensure that our deformations are unramified over $\Q(\mu_{p^{\infty}})$ at all primes $\mathfrak{p}|p$.
\begin{Def}\label{extraordinary}
Let $\mathcal{F}_p:\mathcal{C}_{\Z_p}\rightarrow \text{Sets}$ be the functor of deformations of $\bar{\rho}$ which consist of a sum of two characters, we refer deformations $\mathcal{F}_p$ as \textit{diagonal}. In greater detail, $\mathcal{F}_p(R)$ consists of deformations $\rho_R:G_p\rightarrow \text{GL}_2(R)$ of $\bar{\rho}$ such that 
\begin{itemize}
\item $\rho_{\restriction G_p}\simeq \varphi_1\oplus \varphi_2$ where $\varphi_1$ and $\varphi_2$ are two characters (possibly ramified)
\item $\det \rho=\chi^{i+p^2(p-1)}$.
\end{itemize}The tangent space $\mathcal{N}_p$ is defined as the set of deformations of $\bar{\rho}$ to the dual numbers $\mathcal{F}_p(\F_p[\epsilon])$ which has a natural structure of an $\F_p$ vector space and is realized as a subspace of $H^1(G_p, Ad^0\bar{\rho})$.
\end{Def} 
\par The functor $\mathcal{F}_p$ is a liftable deformation condition. Observe that since $G_p$ acts diagonally, the $1$ dimensional space $\text{Diag}(\bar{\rho}_{\restriction G_p})$ is a summand of $Ad^0\bar{\rho}_{\restriction G_p}$ and \[\A_{\restriction G_p}\simeq \F_p(\bar{\chi}_{\restriction G_p}^i)\oplus \text{Diag}(\bar{\rho}_{\restriction G_p})\oplus \F_p(\bar{\chi}_{\restriction G_p}^{-i}).\]
\begin{Def}\label{liftablebalanced}
Let $\mathcal{F}$ be a deformation condition for $\bar{\rho}_{\restriction G_p}$ with tangent space $\mathcal{N}_{\mathcal{F}}$. 
\begin{itemize}
    \item The deformation condition $\mathcal{F}$ is \textit{liftable} if for every small extension $f:R\rightarrow S$, the induced map $f_*:\mathcal{F}(R)\rightarrow \mathcal{F}(S)$ is surjective.
    \item The deformation condition $\mathcal{F}$ is \textit{balanced} if
    \[\dim \mathcal{N}_p=\dim H^0(G_p, \text{Ad}^0\bar{\rho})+1=2.\]
\end{itemize}
\end{Def}
\begin{Prop}\label{tangentequality}The deformation condition $\mathcal{F}_p$ is liftable and balanced.
\end{Prop}
\begin{proof} That $\mathcal{F}_p$ is liftable follows from Lemma 4.5 of \cite{blondeau}. We show that it is balanced, i.e. we deduce that \[\dim \mathcal{N}_p=\dim H^0(G_p, \text{Ad}^0\bar{\rho})+1=2.\]Explicitly, $\mathcal{N}_p$ is the set of elements $X\in H^1(G_p, Ad^0\bar{\rho})$ for which the twist $\bar{\rho}(Id+\epsilon X)$ is diagonal, we deduce that $\mathcal{N}_p=H^1(G_p, \text{Diag}(\bar{\rho}))$.
\par Since $\bar{\chi}_{\restriction G_p}^i\neq 1$, 
\[\begin{split}&\dim H^0(G_p, \A)\\
&=\dim H^0(G_p, \F_p(\bar{\chi}^i))+\dim H^0(G_p, \F_p)+\dim H^0(G_p, \F_p(\bar{\chi}^{-i}))\\&=1.\\
\end{split}\]
\par We employ the Euler characteristic formula and local duality to compute $\dim \mathcal{N}_p$
\[\begin{split}\dim \mathcal{N}_p&=\dim H^1(G_p, \text{Diag}(\bar{\rho}))\\
&=1+\dim H^0(G_p, \text{Diag}(\bar{\rho}))+\dim H^2(G_p, \text{Diag}(\bar{\rho}))\\
&=1+\dim H^0(G_p, \F_p)+\dim H^0(G_p, \F_p(\bar{\chi}))\\
&=\dim H^0(G_p, \A)+1=2.\\
\end{split}\]
\end{proof}
\par The assumptions on $\bar{\rho}$ satisfy those on the assumptions laid out on the residual representations examined in \cite{hamblenramakrishna}.
\begin{Lemma}\label{firstlemma} The conditions stipulated in \cite[Theorem 2]{hamblenramakrishna} are satisfied by the representation $\bar{\rho}$.
\end{Lemma}
\begin{proof} The reducible representation $\bar{\rho}=\mtx{\phi}{\ast}{0}{1}$ with $\phi=\bar{\chi}^i$, we recall that $i\neq (p-1)/2$ is odd and $2\leq i\leq p-3$. We enumerate the six conditions and show that they are satisfied:
\begin{itemize}
\item Condition (0) is satisfied since $p\neq 2$.
\item Condition (1) requires that $\bar{\rho}$ is indecomposable (or in other words, not semi-simple), this follows by construction as the extension $E/\Q(\mu_p)$ is a non-trivial extension.
\item Condition (2) requires that $\phi^2\neq 1$, or equivalently, $2i\not\equiv 0$ mod $(p-1)$ which is satisfied.
\item Condition (3) requires that $\phi\neq \bar{\chi}^{\pm 1}$ which follows from the assumption on $i$.
\item Condition (4) is automatically satisfied since the field in question $\F_q$ is $\F_p$.
\item Since $i$ is odd, $\bar{\rho}$ is odd. Condition (5) requires that $\bar{\rho}$ that $\bar{\rho}_{\restriction G_p}$ is not unramified of the form $\mtx{1}{\ast}{0}{1}$ (where $\ast$ may be trivial). The integer $i$ is divisible by $p-1$ and as a consequence, this condition is satisfied.
\end{itemize}
\end{proof}
In the proof of Theorem $\ref{1.1}$, the following Lemma is applied in deducing that the image of a suitably constructed lift $\rho$ of $\bar{\rho}$ contains the principal congruence subgroup of $\text{SL}_2(\Z_p)$.
\begin{Lemma}\label{3.4}
Let $p\geq 5$ be a prime, $X$ be a closed subgroup of $\text{SL}_2(\Z_p)$ and let $X_2$ be the image of $X$ in $\text{SL}_2(\Z/p^2\Z)$. Suppose that $X_2$ contains the principal congruence subgroup of $\text{SL}_2(\Z/p^2\Z)$, then $X$ contains the principal congruence subgroup of $\text{SL}_2(\Z_p)$.
\end{Lemma}
\begin{proof}
The proof follows from that of Lemma 3 in \cite[Chapter 4, Section 3.4]{serre} with very little modification.
\end{proof}
\begin{proof}(of Theorem $\ref{1.1}$)
\par This result shall follow from the main result of \cite{hamblenramakrishna} after a single modification is made to their construction. Their method relies on the existence of a balanced liftable deformation condition at each prime at which the residual Galois representation is allowed to ramify. There are cohomological obstructions to lifting a residual Galois representation which satisfies these local conditions. More specifically if a certain \textit{Selmer group} does vanish the local to global deformation theoretic construction can be applied. On adjoining some auxiliary deformation conditions at a finite set of \textit{trivial primes} (cf. Definition 12 in \cite{hamblenramakrishna}) the associated Selmer group can be shown to vanish (cf. Proposition 46 in \cite{hamblenramakrishna}). Hamblen and Ramakrishna work throughout with the ordinary deformation condition (cf. \cite{RamFM} and \cite{Taylor}) at $p$. Instead, we shall prescribe the diagonal deformation condition which is also a balanced and liftable deformation. This was established in Proposition $\ref{tangentequality}$. The construction of Hamblen and Ramakrishna does not in any specific way utilize the ordinary deformation condition and the diagonal deformation condition may as well be used in its place since this condition is also liftable and balanced. It follows that there are infinitely many characteristic zero deformations $\rho$ satisfying the conditions of the main theorem.
\par In greater detail, it is a consequence of Proposition 42 of \cite{hamblenramakrishna} that \[\text{image}\lbrace\rho(G_{\Q(\mu_{p^{\infty}})})\rightarrow \text{SL}_2(\Z/p^2\Z)\rbrace\] contains the principal congruence subgroup. It follows from Lemma $\ref{3.4}$ that $\rho(G_{\Q(\mu_{p^{\infty}})})$ contains the principal congruence subgroup in $\text{SL}_2(\Z_p)$.
\end{proof}
\begin{proof}(of Corollory \ref{main})
\par Let $\rho$ be a lift of $\bar{\rho}$ satisfying the conditions of Theorem $\ref{main}$. We let $L$ be the fixed field of $\rho$. The infinite cyclotomic field $\Q(\mu_{p^{\infty}})$ is the fixed field of $\text{det}\rho=\chi^{i+p^2(p-1)}$. Since $\rho(G_{\Q(\mu_{p^{\infty}})})$ contains the principal congruence subgroup of $\text{SL}_2(\Z_p)$ we see that $\text{Gal}(L/\Q(\mu_{p^{\infty}}))$ is topologically isomorphic to a subgroup of $\text{SL}_2(\Z_p)$ which contains the principal congruence subgroup.
\par The local representation $\rho_{\restriction G_p}$ is a sum of characters $\varphi_1$ and $\varphi_2$, we deduce that $L$ is unramified at all primes above $p$.
\par Since $\bar{\rho}_{\restriction G_{\Q(\mu_p)}}$ is unramified at all primes, it follows that at the primes at which $L$ is ramified, $L$ must be tamely ramified.
\end{proof}
\begin{proof}(of Proposition $\ref{lastprop}$)
\par Indeed if $p$ is such a prime, since $p$ is irregular, the quotient $\mathcal{C}=\text{Cl}(\Q(\mu_p))\otimes \F_p\neq 0$. By assumption each even eigenspace $\mathcal{C}(\bar{\chi}^{2j})=0$. Thus there exists an odd integer $1\leq i\leq p-2$ for which $\mathcal{C}(\bar{\chi}^{i})\neq 0$. Since $p\equiv 1\mod{4}$, and $i$ is odd, we have that $i\neq \frac{p-1}{2}$. On the other hand, $i\neq 1$ since $\mathcal{C}(\bar{\chi})=0$ (cf. Proposition 6.16 of \cite{washington}). By Herbrand's theorem implies that $p$ divides the numerator of the Bernoulli number $B_{p-i}$. Since $B_2=\frac{1}{6}$ we deduce that $i\neq p-2$ and thus lies in the range $2\leq i\leq p-3$.
\end{proof}

\end{document}